\documentclass[11pt,a4paper]{amsart}
\usepackage{graphicx,multirow,array,amsmath,amssymb}

\newtheorem{theorem}{Theorem}

\begin{document}

\title{Stick number of spatial graphs}

\author[M. Lee]{Minjung Lee} 
\address{Department of Mathematics, Korea University, Seoul 02841, Korea}
\email{mjmlj@korea.ac.kr}
\author[S. No]{Sungjong No}
\address{Department of Statistics, Ewha Womans University, Seoul 03760, Korea}
\email{sungjongno84@gmail.com}
\author[S. Oh]{Seungsang Oh}
\address{Department of Mathematics, Korea University, Seoul 02841, Korea}
\email{seungsang@korea.ac.kr}

\thanks{Mathematics Subject Classification 2010: 57M25, 57M27}
\thanks{The second author was supported by the BK21 Plus Project through the National Research Foundation of 
Korea (NRF) grant funded by the Korean Ministry of Education (22A20130011003).}
\thanks{The corresponding author(Seungsang Oh) was supported by the National Research Foundation of Korea(NRF) grant funded by the Korea government(MSIP) (No. NRF-2017R1A2B2007216).}

\maketitle

\begin{abstract}
For a nontrivial knot $K$,
Negami found an upper bound on the stick number $s(K)$  
in terms of its crossing number $c(K)$ which is $s(K) \leq 2 c(K)$. 
Later, Huh and Oh utilized the arc index $\alpha(K)$ to present a more precise upper bound
$s(K) \leq \frac{3}{2} c(K) + \frac{3}{2}$.
Furthermore, Kim, No and Oh found an upper bound on the equilateral stick number $s_{=}(K)$ 
as follows; $s_{=}(K) \leq 2 c(K) +2$.
As a sequel to this research program,
we similarly define the stick number $s(G)$ and the equilateral stick number $s_{=}(G)$ of a spatial graph $G$,
and present their upper bounds as follows;
$$ s(G) \leq \frac{3}{2} c(G) + 2e + \frac{3b}{2} -\frac{v}{2}, $$ 
$$ s_{=}(G) \leq 2 c(G) + 2e + 2b - k, $$
where $e$ and $v$ are the number of edges and vertices of $G$, respectively,
$b$ is the number of bouquet cut-components,
and $k$ is the number of non-splittable components.
\end{abstract}

\section{Introduction} \label{sec:intro}

Throughout this paper we work in the piecewise linear category.
A graph is a finite set of vertices connected by edges allowing loops and multiple edges.
A {\em spatial graph\/} is a graph embedded in $\mathbb{R}^3$.
We consider two spatial graphs to be the same if they are equivalent under ambient isotopy.
A {\em bouquet\/} is a spatial graph consisting of only one vertex and loops. 
Note that a knot is a spatial graph consisting of a vertex and a loop. 

A {\em stick spatial graph\/} is a spatial graph which consists of finite line segments, 
called {\em sticks}, as drawn in Figure~\ref{fig:stick}.
This presentation of spatial graphs can be considered to be a reasonable mathematical model
of polymers because such physical objects have rigidity.
Concerning stick spatial graphs, one natural problem may be to determine the number of sticks.
The {\em stick number\/} $s(G)$ of a spatial graph $G$ is defined to be 
the minimal number of sticks required to realize the spatial graph as a stick spatial graph.

\begin{figure}[h]
\includegraphics{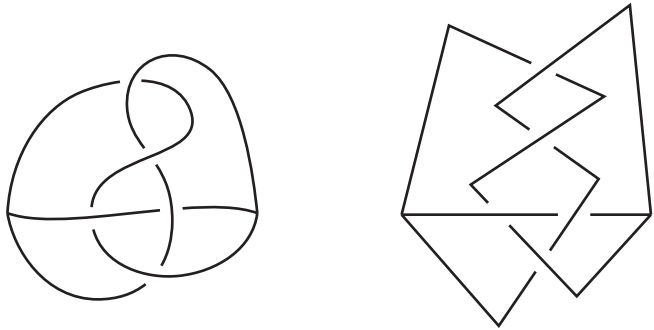}
\caption{Spatial $\theta$-curve $5_1$ and its stick presentation}
\label{fig:stick}
\end{figure}

The stick number of a nontrivial knot $K$ is related to its crossing number $c(K)$ 
by the following inequalities~\cite{Ca, HO1, Ne};
$$\frac{7+\sqrt{8c(K)+1}}{2} \leq s(K) \leq \frac{3}{2} c(K) + \frac{3}{2}.$$
Particularly, a 2-bridge knot $K$ with $c(K) \geq 6$ has a better upper bound 
$s(K) \leq c(K) + 2$ shown in~\cite{FLS, HNO, Mc}.
There are few knots whose stick number can be determined exactly.
Knots with small crossing numbers were investigated by Randell~\cite{Ra}. 
Jin~\cite{Ji} determined the precise stick number of a $(p,q)$-torus knot $T(p,q)$ 
in case of $2 \leq p \leq q \leq 2p$ as $s(T(p,q)) = 2q$.
The same result was found independently by Adams {\em et al.\/}~\cite{ABGW}, 
but for a smaller range of parameters.

As a sequel to this research program, we consider the stick number of spatial graphs.
We consider two types of 2-spheres that separate a spatial graph $G$ into two parts.
Such a 2-sphere is called a splitting-sphere if it does not meet $G$,
and a cut-sphere if it intersects $G$ in a single vertex which is called a cut vertex.
We maximally decompose $G$ into {\em cut-components\/} by cutting $G$ 
along a maximal set $\mathcal{S}$ of splitting-spheres and cut-spheres 
where any two spheres are either disjoint or intersect each other in a cut vertex.
$G$ is said to be {\em non-splittable\/} if there is no splitting-sphere separating $G$.
Each cut-component  is non-splittable.
In particular, if such a cut-component is a bouquet, it is called a {\em bouquet cut-component\/}.
The crossing number $c(G)$ is the minimal number of double points 
in any generic projection of $G$ into the plane $\mathbb{R}^2$.
Here, these double points are disjoint from the projected vertices of the spatial graph.

\begin{theorem} \label{thm:stick}
Let $G$ be any spatial graph with $e$ edges, $v$ vertices and $b$ bouquet cut-components.
Then 
$$ s(G) \leq \frac{3}{2} c(G) + 2e + \frac{3b}{2} - \frac{v}{2}. $$
\end{theorem}

For a knot $K$, this result implies $s(K) \leq \frac{3}{2} c(K) + 3$ 
which is very close to the previously known upper bound.
Furthermore for both the unlink with $n$ components and the trivial $\theta_n$-curve even for any large $n$,
the stick number $s(G)$ is actually equal to the value $2e + \frac{3b}{2} - \frac{v}{2}$ of $G$ (obviously $c(G)=0$).
This implies that this sum of the last three terms is best possible.

We are also interested in another quantity concerning stick spatial graphs.
An {\em equilateral spatial graph\/} is a stick spatial graph which consists of uniform length sticks.
The {\em equilateral stick number\/} $s_{=}(G)$ of a spatial graph $G$ is defined to be 
the minimal number of sticks required to construct an equilateral spatial graph representation of $G$.

Even for knots, little is known about the equilateral stick number.
Rawdon and Scharein~\cite{RS} used algorithms in the software KnotPlot to compute
upper bounds for the equilateral stick number of all prime knots with up to 10 crossings.
They showed that all such knots except seven can be constructed with 
the same number of equal length sticks as their stick numbers.
Recently, Kim, No and Oh~\cite{KNO} found an upper bound for a nontrivial knot $K$ as follows;
$$ s_{=}(K) \leq 2 c(K) +2. $$

\begin{theorem} \label{thm:equilateral}
Let $G$ be any spatial graph with $e$ edges and $b$ bouquet cut-components.
Suppose that $G$ has $k$ non-splittable components.
Then 
$$ s_{=}(G) \leq 2 c(G) + 2e + 2b - k. $$
\end{theorem}

For a nontrivial knot $K$, this result implies $s_{=}(K) \leq 2 c(K) + 3$ 
which is very close to the previously known upper bound.
Furthermore for both the unlink with $n$ components and the trivial $\theta_n$-curve even for any large $n$,
the stick number $s_{=}(G)$ is equal to the value $2e + 2b - k$ of $G$,
implying that this sum of the last three terms is best possible.

\section{Arc index of spatial graphs} \label{sec:arc}

In an {\em arc presentation\/} of a spatial graph $G$,
$G$ is constructed in an open-book with finitely many half-planes so that 
it meets each half-plane in exactly one simple arc with two different end-points on the binding axis.
Therefore the binding axis contains all vertices of $G$ and 
finitely many points from the interiors of edges of $G$,
and each edge of $G$ may pass from one page to another across the binding axis.
See Figure~\ref{fig:arc} for an arc presentation of the $\theta$-curve $5_1$ with eight pages.
The {\em arc index\/} $\alpha(G)$ is defined to be the minimal number of
pages among all possible arc presentations of $G$.

\begin{figure}[h]
\includegraphics{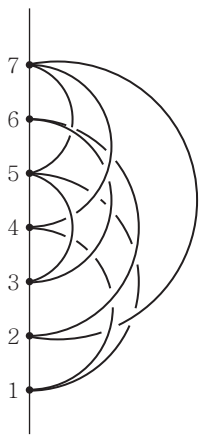}
\caption{Arc presentation of $5_1$}
\label{fig:arc}
\end{figure}

Bae and Park~\cite{BP} established an upper bound on the arc index of a nontrivial knot 
in terms of the crossing number.
The authors~\cite{LNO} generalized the argument used in~\cite{BP} 
to find an upper bound on the arc index of a spatial graph as follows.
This upper bound result is crucial in the proofs of our two main theorems.

\begin{theorem} {\textup{(Lee-No-Oh~\cite{LNO})}} \label{thm:LNO}
Let $G$ be any spatial graph with $e$ edges and $b$ bouquet cut-components. 
Then 
$$\alpha(G) \leq c(G)+e+b.$$
\end{theorem}

\section{Circular arc presentations} \label{sec:circular}

We introduce a variant of an arc presentation of a spatial graph $G$,
called a {\em circular arc presentation\/},
which is more useful to construct a stick presentation of $G$ in Section~\ref{sec:stick}.
First, take a circular disk $D$ in $\mathbb{R}^2$
whose boundary actually indicates the binding axis combined with point at infinity.
Now draw straight chords instead of arcs of the arc presentation.
During this procedure all the page numbers and the under/over crossings are preserved.
The chords are denoted by $l_1, \dots, l_n$ in the order of the page numbers,
so that if $l_i$ and $l_j$ share a crossing in the interior of $D$ and $i<j$, then $l_i$ passes under $l_j$.
Note that $\alpha(G)=n$.
The points of the graph lying on the boundary of $D$ are called the {\em binding points\/}. 
Note that each chord has two different binding points at its ends.
See Figure~\ref{fig:circular}.

\begin{figure}[h]
\includegraphics{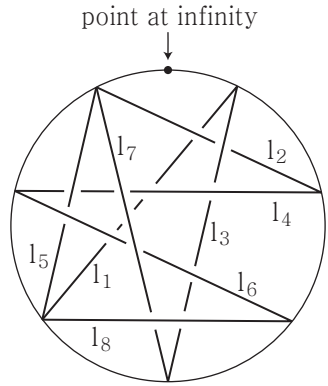}
\caption{Circular arc presentation of $5_1$}
\label{fig:circular}
\end{figure}

For a binding point $b$, the {\em initiating page number\/} $p(b)$ is
the smallest page number among all chords incident to $b$.
A chord $l_i$ is said to have an {\em initiating end\/} 
if this end point has the initiating page number $i$.
Then all chords are distinguished into three types as follows;
a {\em bi-initiating chord\/} if it has two initiating ends on both sides,
a {\em uni-initiating chord\/} if it has an initiating end only on one side, and
a {\em non-initiating chord\/} if it has no initiating ends.
In Figure~\ref{fig:circular}, 
the initiating page numbers of the seven binding points in clockwise order 
starting at point at infinity are 1, 2, 6, 3, 1, 4, and 2.
Furthermore, $l_1$ and $l_2$ are bi-initiating chords, $l_3$, $l_4$ and $l_6$ are uni-initiating chords, 
and the others are non-initiating chords.

\section{Stick presentations} \label{sec:stick}

In this section we will prove Theorem~\ref{thm:stick}.

\begin{proof}[Proof of Theorem \ref{thm:stick}]
Let $G$ be a spatial graph with $e$ edges and $v$ vertices, and 
$b$ is the number of bouquet cut-components of $G$.
Assume that $\alpha(G)=n$ and $C$ is its circular arc presentation with $n$ chords $l_1, \dots, l_n$. 
Further assume that $C$ has $n_2$, $n_1$ and $n_0$ numbers of bi-initiating, uni-initiating and 
non-initiating chords where $n_2 + n_1 + n_0 = n$.

The projection $\pi : \mathbb{R}^3 \rightarrow \mathbb{R}^2$ is defined by $\pi(x,y,z)=(x,y)$. 
For each $i = 1, \dots, n$, put a horizontal line segment $h_i$ in $D \times \{i\}$ so that $\pi(h_i)=l_i$
as in Figure~\ref{fig:stack}. 

\begin{figure}[h]
\includegraphics{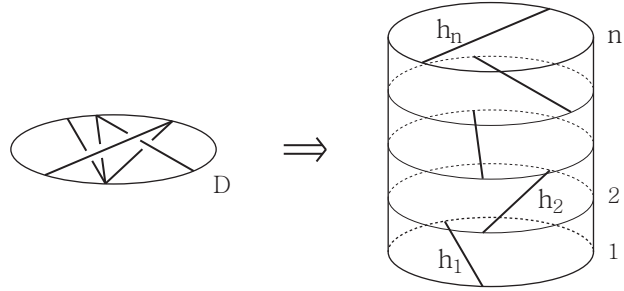}
\caption{Simple stack of $n$ chords} 
\label{fig:stack}
\end{figure}

Now, deform these line segments $h_i$ in an inductive manner on $i$, 
meanwhile their projections are preserved.
The segment $h_1$ stays in itself and so its highest $z$-coordinate is $z_1=1$.
Assume that $h_1, h_2, \dots, h_{k-1}$ are deformed inductively
with their highest integer-valued $z$-coordinates $z_1 < z_2 < \cdots < z_{k-1}$, respectively.

Deforming $h_k$ depends on the type of the chord $l_k$ as in Figure~\ref{fig:cases}.
If $l_k$ is a bi-initiating chord, raise up $h_k$ into $D \times \{ z_{k-1} \! + \! 1 \}$
with its highest $z$-coordinate $z_k = z_{k-1} \! + \! 1$.
If $l_k$ is a uni-initiating chord, 
then exactly one end $b$ of $l_k$ has the initiating page number $p(b) = k$.
Let $b'$ denote the other end with its initiating page number $p(b') = j' < k$.
Raise up $h_k$ obliquely so that its one end is $h_{j'} \cap \pi^{-1}(b')$
and the other end is $b \times \{ z_k \}$ for a sufficiently large integer $z_k > z_{k-1}$.
More precisely, the interior of the triangle between this new $h_k$ and 
a horizontal line segment $H$ in $D \times \{z_k\}$ with $\pi(H)=l_k$
has no intersection with any other $h_i$, $i < k$.
Lastly, if $l_k$ is a non-initiating chord, 
then both ends $b$ and $b'$ of $l_k$ have the initiating page numbers 
$p(b) = j < k$ and $p(b') = j' < k$.
Bend $h_k$ at the center $c \times \{ k \}$ into two connected line segments and 
raise up $h_k$ so that its ends are $h_{j} \cap \pi^{-1}(b)$ and $h_{j'} \cap \pi^{-1}(b')$,
and the center point goes to $c \times \{ z_k \}$ for a sufficiently large integer $z_k > z_{k-1}$.
Eventually the interiors of the two triangles between this new $h_k$ and 
a horizontal line segment $H$ in $D \times \{z_k\}$ with $\pi(H)=l_k$
have no intersection with any other $h_i$, $i < k$.

\begin{figure}[h]
\includegraphics{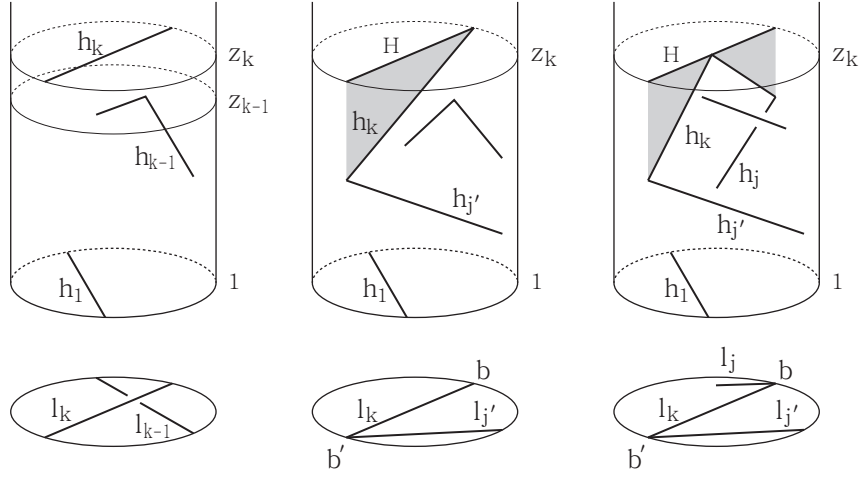}
\caption{Bi-initiating chord, uni-initiating chord and non-initiating chord cases} 
\label{fig:cases}
\end{figure}

Then the new stick spatial graph $G' = \cup_{i=1}^{n} h_i$ in $D \times [1,\infty)$
is isotopic to $G$ and consists of $n_2 + n_1 + 2 n_0 = n + n_0$ sticks as desired.
Therefore,
$$ s(G) \leq \alpha(G) + n_0. $$

On the other hand, split all chords in $C$ slightly away from binding points.
Then there are total $2n$ ends of the $n$ chords.
Note that the number of binding points is $n-e+v$ by an easy combinatorial calculation.
Since each binding point is an initiating end of exactly one of chords incident to it,
there are exactly $n-e+v$ initiating ends among these $2n$ ends.
Thus $n+e-v$ ends are not initiating.
Since non-initiating chords have no initiating ends,
the number $n_0$ of non-initiating chords is at most $\frac{n+e-v}{2}$. 
This implies
$$s(G) \leq \frac{3}{2} \alpha(G) + \frac{e}{2} - \frac{v}{2}.$$

Now we complete the proof by applying Theorem~\ref{thm:LNO}.
\end{proof}

\section{Equilateral stick presentations} \label{sec:equilateral}

In this section we will prove Theorem~\ref{thm:equilateral}.

\begin{proof}[Proof of Theorem \ref{thm:equilateral}]

We will follow the main argument in the proof of Theorem~1.1 in~\cite{KNO}.
Let $G$ be a spatial graph with $e$ edges and $b$ bouquet cut-components.
Suppose that $G$ has $k$ non-splittable components.
Let $G_1$ be a non-splittable component of $G$.

Assume that $\alpha(G_1)=n$ and $A$ is its arc presentation with $n$ arcs. 
Let $M$ be a sufficiently large number.
Now we stretch $A$ so that each arc is deformed to two connected sticks of length $M$ each
on the related page, while all binding points are fixed, as drawn in Figure~\ref{fig:equilateral}.
Already we obtain a equilateral spatial graph of $G_1$ consisting of $2 n$ sticks of length $M$.
Therefore,
$s_{=}(G_1) \leq 2 \alpha(G_1)$.

\begin{figure}[h]
\includegraphics{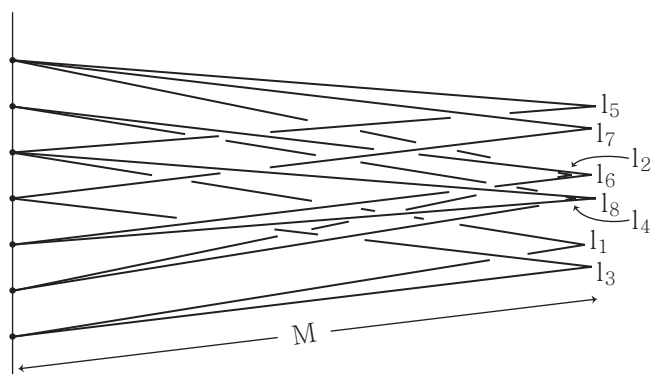}
\caption{Equilateral stick presentation of $5_1$} 
\label{fig:equilateral}
\end{figure}

Now we will reduce one more stick from $A$.
Assume that the top binding point has $m$ sticks incident to, say $d_1, \dots, d_m$.
Let $e_1, \dots, e_m$ denote the other $m$ sticks connected to $d_1, \dots, d_m$, respectively.
As in Figure~\ref{fig:reducing}, delete all $m$ sticks $d_1, \dots, d_m$.
Rotate the stick $e_1$ counterclockwise along its endpoint at the binding axis
on the related page until $e_1$ is close enough to the binding axis.
Now rotate each $e_i$, $i=2,\dots,m$, counterclockwise along its endpoint at the binding axis
on the related page so that the distance between the other endpoints of $e_1$ and $e_i$
which are not at the binding axis is $M$.
Glue a new stick $f_i$ of length $M$ to these endpoints as a substitution for $d_i$.
Clearly this formation does not change the spatial graph type,
but we reduce one stick from the original.
Hence the resulting equilateral spatial graph $G_1$ satisfies
$$ s_{=}(G_1) \leq 2 \alpha(G_1) - 1. $$

\begin{figure}[h]
\includegraphics{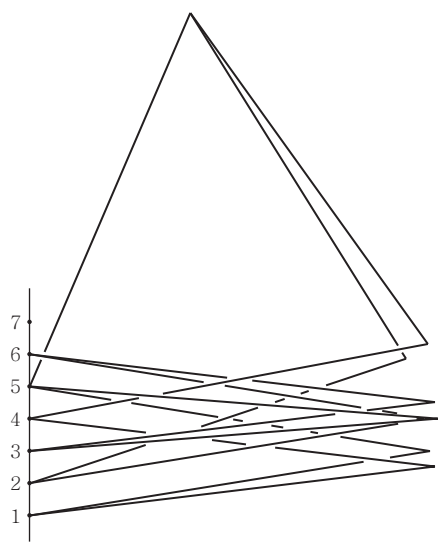}
\caption{Reducing one more stick} 
\label{fig:reducing}
\end{figure}

Furthermore, if $G$ can be split into two components $G_1$ and $G_2$,
then obviously $s_{=}(G) = s_{=}(G_1) + s_{=}(G_2)$.
Theorem \ref{thm:equilateral} follows directly from Theorem~\ref{thm:LNO}.

Note that unlike the case of nontrivial knots in~\cite{KNO},
it is possible that there is an arc connecting the top and the bottom binding points,
for example, the trivial $\theta_n$-curve.
We cannot repeat the same argument at the bottom binding point.
\end{proof}

\end{document}